\documentclass{amsart}

\usepackage{amscd}
\newcommand{\darkrad}{0.115}

\newcommand{\Q}{{\mathbb{Q}}}        % the reationals
\newcommand{\R}{{\mathbb{R}}}        % the reals
\newcommand{\F}{{\mathbb{F}}}        % finite fields
\newcommand{\Z}{{\mathbb{Z}}}        % the integers
\newcommand{\Zm}[1]{\Z/{#1}\Z}

\newcommand{\e}{\varepsilon}

\newcommand{\la}{\lambda}
\newcommand{\D}{\Delta}

\newcommand{\oddots}{{\mathinner{\mkern1mu\raise1pt\vbox{\kern7pt\hbox{.}}\mkern2mu\raise4pt\hbox{.}\mkern2mu\raise7pt\hbox{.}\mkern1mu}}}

\newcommand{\s}{\sigma}

\newcommand{\ksep}{k_{{\mathrm{sep}}}}

\newcommand{\pform}[1]{{\langle\!\langle{#1}\rangle\!\rangle}} % a Pfister form
\DeclareMathOperator{\Spin}{Spin}           % the Spin group

\DeclareMathOperator{\SL}{SL}

\DeclareMathOperator{\PGL}{PGL}
\newcommand{\SO}{\mathrm{SO}}

\newcommand{\Gm}{\mathbb{G}_m}

\newcommand{\dE}{\ensuremath{^2\!E_6}}

\newcommand{\iiiD}{^3\!D_4}
\newcommand{\viD}{^6\!D_4}

\newcommand{\iiD}{^2\!D_4}

\DeclareMathOperator{\Spec}{Spec}

\DeclareMathOperator{\Gal}{Gal}  

\DeclareMathOperator{\im}{im}

\DeclareMathOperator{\chr}{char}

\DeclareMathOperator{\res}{res}

\DeclareMathOperator{\aut}{Aut}
\DeclareMathOperator{\Aut}{Aut}

\newcommand{\Hom}{{\mathrm{Hom}}}

\newcommand{\iso}{\xrightarrow{\sim}}

\newcommand{\ra}{\rightarrow}

\usepackage{amsthm}

\newtheorem{thm}[equation]{Theorem}
\newtheorem{lem}[equation]{Lemma}

\newtheorem{prop}[equation]{Proposition}

\newtheorem*{thm*}{Theorem}
\newtheorem*{prop*}{Proposition}
\newtheorem*{cor*}{Corollary}
\newtheorem*{lem*}{Lemma}
\newtheorem*{MT*}{Main Theorem}

\theoremstyle{definition} %
\newtheorem{defn}[equation]{Definition}
\newtheorem*{defn*}{Definition}

\newtheorem{eg}[equation]{Example}

\theoremstyle{remark} %

\newtheorem*{rmk*}{Remark}
\newtheorem*{rmks*}{Remarks}

\newtheoremstyle{exercise}% name
  {3pt}%      Space above
  {3pt}%      Space below
  {\small}%         Body font
  {}%         Indent amount (empty = no indent, \parindent = para indent)
  {\sc\small}% Thm head font
  {.}%        Punctuation after thm head
  {.5em}%     Space after thm head: " " = normal interword space;
   {}     %       \newline = linebreak
  {}%         Thm head spec (can be left empty, meaning `normal')

\theoremstyle{exercise}

\makeatletter
{\renewcommand{\theequation}{#1}}%
{\renewcommand{\theequation}{\arabic{equation}}\addtocounter{equation}{-1}\global\@ignoretrue}
\makeatother

\makeatletter
{\renewcommand{\theequation}{#1}\begin{eqnarray}}%
{\end{eqnarray}\renewcommand{\theequation}{\arabic{equation}}\addtocounter{equation}{-1}\global\@ignoretrue}
\makeatother

\makeatletter
{\smallskip \refstepcounter{equation}\noindent{\textbf{\theequation.} }{{\textbf{#1.}}}}%
{\smallskip \global\@ignoretrue}
\makeatother

\makeatletter
{\smallskip \refstepcounter{equation}\noindent{\textbf{\theequation.} }{{\textbf{#1}}}}%
{\smallskip \global\@ignoretrue}
\makeatother

\makeatletter
{\smallskip \refstepcounter{equation}{\sc \theequation}{\sc (#1).}}%
{\smallskip \global\@ignoretrue}
\makeatother

\makeatletter
{\smallskip\refstepcounter{equation}\noindent{\textbf{\theequation.}}{\textsl{ #1.}}}%
{\smallskip\global\@ignoretrue}
\makeatother

\makeatletter
\newenvironment{borel*}%
{\smallskip \refstepcounter{equation}\noindent{\textbf{\theequation.}}}%
{\global\@ignoretrue}
\makeatother

\newcommand{\flist}[1]{\hangindent\leftmargini\textup{(1)}\hskip\labelsep {#1}%
\begin{enumerate}%
\setcounter{enumi}{1}%
}

\usepackage[all]{xy}

\theoremstyle{plain}

\renewcommand{\theequation}{\arabic{equation}}

\newcommand{\even}{{\mathrm{even}}}

\newcommand{\sep}{{\mathrm{sep}}}

\renewcommand{\O}{\mathrm{O}}

\newcommand{\Gb}{\overline{G}}

\newcommand{\Gt}{\widetilde{G}}
\newcommand{\Tt}{\widetilde{T}}

\DeclareMathOperator{\PSO}{PSO}

\begin{document}

\title[Outer automorphisms of algebraic groups]{Outer automorphisms of algebraic groups
and determining groups by their maximal tori}

\begin{abstract}
We give a cohomological criterion for existence of outer automorphisms of a semisimple algebraic group over an arbitrary field.  This criterion is then applied to the special case of groups of type $D_{2n}$ over a global field, which completes some of the main results from 
the paper ``Weakly commensurable arithmetic groups and isospectral locally symmetric spaces'' (Pub. Math. IHES, 2009) by Prasad and Rapinchuk and gives a new proof of a result from another paper by the same authors.
\end{abstract}


%\date{\tt{Version of \today.}}

\author{Skip Garibaldi}
\address{Department of Mathematics and Computer Science, Emory University, Atlanta, GA 30322, USA}
\email{skip@mathcs.emory.edu}

\maketitle

One goal of this paper is the (rather technical) Theorem \ref{MT} below, which completes some of the main results in the remarkable paper 
\cite{PrRap:weakly} by Gopal Prasad and Andrei Rapinchuk.  For example, combining their Theorem 7.5 with our Theorem \ref{MT} gives:
\begin{thm} \label{tori}
Let $G_1$ and $G_2$ be connected absolutely simple algebraic groups over a number field $K$ that have the same $K$-isomorphism classes of maximal $K$-tori.  Then:
\begin{enumerate}
\item $G_1$ and $G_2$ have the same Killing-Cartan type (and even the same quasi-split inner form) or one has type $B_n$ and the other has type $C_n$.
\item \label{tori2} If $G_1$ and $G_2$ are isomorphic over an algebraic closure of $K$ and they are not of type $A_n$ for $n \ge 2$, $D_{2n+1}$, or $E_6$, then $G_1$ and $G_2$ are $K$-isomorphic.
\end{enumerate}
\end{thm}

This result is essentially proved by Prasad-Rapinchuk in \cite{PrRap:weakly}, except that paper omits 
types $D_{2n}$ for $2n \ge 4$ in \eqref{tori2}.  Our Theorem \ref{MT} gives a new proof of the $2n \ge 6$ case (treated by Prasad-Rapinchuk in a later paper \cite[\S9]{PrRap:Deven}) and settles the last remaining case of groups of type $D_4$.  Note that in Theorem \ref{tori}\eqref{tori2},
types $A_n$, $D_{2n+1}$, and $E_6$ are genuine exceptions by \cite[7.6]{PrRap:weakly}.  

Similarly, combining our Theorem \ref{MT} with the arguments in \cite{PrRap:weakly} implies that their Theorems 4, 8.16, and 10.4 remain true if you delete ``$D_4$" from their statements---that is, the conclusions of those theorems regarding weak commensurability, locally symmetric spaces, etc., also hold for groups of type $D_4$.

We mention the following specific result as an additional illustration.  For a Riemannian manifold $M$, write $\Q L(M)$ for the set of rational multiples of lengths of closed geodesics of $M$.
\begin{thm} \label{hyp.thm}
Let $M_1$ and $M_2$ be arithmetic quotients of real hyperbolic space $\mathbf{H}^n$ for some $n \not\equiv 1 \bmod 4$.  If $\Q L(M_1) = \Q L(M_2)$, then $M_1$ and $M_2$ are commensurable (i.e., $M_1$ and $M_2$ have a common finite-sheeted cover).
\end{thm}
The converse holds with no restriction on $n$, see \cite[Cor.~8.7]{PrRap:weakly}.  The theorem itself holds for $n = 2$ by \cite{Reid:isospectral}; for $n = 3$ by \cite{CHLR}; and for $n = 4, 6, 8, \ldots$ and $n = 11, 15, 19, \ldots$ by \cite[Cor.~8.17]{PrRap:weakly} (which relies on \cite{PrRap:Deven}).  The last remaining case, $n = 7$, follows from Theorem \ref{MT} below and arguments as in \cite{PrRap:weakly}.  The conclusion of Theorem \ref{hyp.thm} is false for $n = 5, 9, 13, \ldots$ by Construction 9.15 in \cite{PrRap:weakly}.

\smallskip
The other goal of this paper is Theorem \ref{flayed.prop}, which addresses the more general setting of a semisimple algebraic group $G$ over an arbitrary field $k$.  That theorem gives a cohomological criterion for the existence of outer automorphisms of $G$, i.e., for the existence of $k$-points on non-identity components of $\aut(G)$.  This criterion and the examples we give of when it holds make up the bulk of the proof of Theorem \ref{MT}, which concerns groups over global fields.

\subsection*{Notation} A \emph{global field} is a finite extension of $\Q$ or $\F_p(t)$ for some prime $p$.  A (non-archimedean) \emph{local field} is the completion of a global field with respect to a discrete valuation, i.e., a finite extension of $\Q_p$ or $\F_p((t))$ for some prime $p$.

We write $H^d(k, G)$ for the $d$-th flat (fppf) cohomology set $H^d(\Spec k, G)$ when $G$ is an algebraic affine group scheme over a field $k$.  In case $G$ is smooth, it is the same as the Galois cohomology set $H^d(\Gal(k), G(\ksep))$ where $\ksep$ denotes a separable closure of $k$ and $\Gal(k)$ denotes the group of $k$-automorphisms of $\ksep$.

We refer to \cite{PlatRap}, \cite{Sp:LAG}, and \cite{KMRT} for general background on semisimple algebraic groups.  Such a group $G$ is an \emph{inner form} of $G'$ if there is a class $\gamma \in H^1(k, \Gb)$, for $\Gb$ the adjoint group of $G$, such that $G'$ is isomorphic to $G$ twisted by $\gamma$.  We write $G_\gamma$ for the group $G$ twisted by the cocycle $\gamma$, following the \TeX-friendly notation of \cite[p.~387]{KMRT} instead of Serre's more logical $_\gamma G$. We say simply that $G$ is \emph{inner} or \emph{of inner type} if it is an inner form of a split group; if $G$ is not inner then it is \emph{outer}.

For a group scheme $D$ of multiplicative type, we put $D^*$ for its dual $\Hom(D, \Gm)$.

%%%%%%%%%%%%%%%%%%%%%%%%%%%%%%%%%%%%%%%%%%
\section{Background: the Tits algebras determine the Tits class} \label{folk.sec}

Fix a semisimple algebraic group $G$ over a field $k$.  Its simply connected cover $\Gt$ and adjoint group $\Gb$ fit into an exact sequence
\begin{equation} \label{std.seq}
\begin{CD}
1 @>>> Z @>>> \Gt @>>> \Gb @>>> 1
\end{CD}
\end{equation}
where $Z$ denotes the (scheme-theoretic) center of $\Gt$. Write $\delta \!: H^1(k, \Gb) \ra H^2(k, Z)$ for the corresponding coboundary map.

There is a unique element $\nu_G \in H^1(k, \Gb)$ such that the twisted group $\Gb_{\nu_G}$ is quasi-split \cite[31.6]{KMRT}, and 
the \emph{Tits class} $t_G$ of $G$ is defined to be $t_G := -\delta(\nu_G) \in H^2(k, Z)$.  The element $t_G$ depends only on the isogeny class of $G$.

For $\gamma \in H^1(k, \Gb)$, the center of the twisted group $\Gt_\gamma$ is naturally identified with (and not merely isomorphic to) $Z$, and a standard twisting argument shows that
\begin{equation}
t_{G_\gamma} = t_G + \delta(\gamma).
\end{equation}

\begin{eg}
If $G$ itself is quasi-split, then $t_G = 0$ and for every $\gamma \in H^1(k, \Gb)$ we have $t_{G_\gamma} = \delta(\gamma)$.
\end{eg}

\begin{defn}[Tits algebras]
A \emph{Tits algebra} of $G$ is an element
\[
\chi(t_G) \in H^2(k(\chi), \Gm) \quad \text{for $\chi \in Z^*$,}
\]
where $k(\chi)$ denotes the subfield of $\ksep$ of elements fixed by the stabilizer of $\chi$ in $\Gal(k)$, i.e., $k(\chi)$ is the smallest separable extension of $k$ so that $\chi$ is fixed by $\Gal(k(\chi))$.

We can modify this definition to replace elements of $Z^*$ with weights.  Fix a pinning for $\Gt$ over $\ksep$ involving a maximal $k$-torus $\Tt$.  As $Z$ is contained in $\Tt$, every weight $\la$---i.e., every $\la \in \Tt^*$---induces by restriction an element of $Z^*$ and we define $\la(t_G)$ to be $\la\vert_Z(t_G)$.  
\end{defn}

Regarding history, the Tits algebras of $G$ were defined in \cite{Ti:R}.  The class $\la(t_G)$ measures the failure of the irreducible representation of $\Gt$ with highest weight $\la$---which is defined over $\ksep$---to be defined over $k$.  Roughly speaking, a typical example of a Tits algebra is provided by the even Clifford algebra of the special orthogonal group of a quadratic form, see for example \cite[\S27]{KMRT}.

Obviously, the Tits class $t_G$ determines the Tits algebras $\chi(t_G)$.  The converse also holds: 

\begin{prop} \label{tits.prop}
If $G$ is absolutely almost simple, then the natural map 
\[
\prod \la \!: H^2(k, Z) \ra \prod H^2(k(\la\vert_Z), \Gm)
\]
is injective, where the products range over minuscule weights $\la$.
\end{prop}

This proposition can probably be viewed as folklore.   I learned it from Alexander Merkurjev and Anne Qu\'eguiner.  Below, we we will use the following restatement: Given classes $\gamma_1, \gamma_2 \in H^1(k, \Gb)$, \emph{if $\la(\delta(\gamma_1)) = \la(\delta(\gamma_2))$ for every minuscule weight $\la$ of $G$, then $\delta(\gamma_1) = \delta(\gamma_2)$.} 

\begin{proof}[Proof of Proposition \ref{tits.prop}]
As $\Gt$ is simply connected, the weight lattice $P$ in the pinning is identified with $\Tt$ and the root lattice $Q$ is the kernel of the restriction $\Tt^* \ra Z^*$.  Since $G$ is assumed absolutely almost simple, its root system is irreducible, and this surjection identifies the minuscule (dominant) weights with the nonzero elements of $Z^*$
\cite[\S{VI.2}, Exercise 5a]{Bou:g4}.  Therefore, the claim is equivalent to showing that the map
\begin{equation} \label{tits.1}
\prod_{\chi \in Z^*} \chi \!: H^2(k, Z) \ra \prod_{\chi \in Z^*} H^2(k(\chi), Z)
\end{equation}
is injective.   This claim depends only on $Z$, so we may replace $G$ with $G_{\nu_G}$ and so assume that $G$ is quasi-split.

We choose the maximal $k$-torus $\Tt$ and pinning in $\Gt$ so that the usual Galois action preserves the fundamental chamber, i.e., permutes the set of simple roots $\D$ and the fundamental dominant weights.  In the short exact sequence $1 \ra Z \ra \Tt \ra \Tt/Z \ra 1$, the set $\D$ is a basis for the lattice $(\Tt/Z)^*$, so $H^1(k, \Tt/Z)$ is zero and the map $H^2(k, Z) \ra H^2(k, \Tt)$ is injective.
 
We fix a set $S$ of representatives of the $\Gal(k)$-orbits in $\D$ and write $\alpha_s$ (resp., $\la_s$) for the simple root (resp., fundamental dominant weight) corresponding to $s \in S$.  Because $\Gt$ is simply connected,
\[
\Tt \cong \prod_{s \in S} R_{k(\la_s\vert_Z)/k}(\im h_{\alpha_s})
\]
where $h_{\alpha_s}$ denotes the homomorphism $\Gm \ra \Tt$ corresponding to the coroot $\alpha^{\vee}_s$ \cite[p.~44, Cor.]{St}---note that the Weil restriction term makes sense because the stabilizers of $\alpha_s$, $\la_s$, and $\la_s\vert_Z$ all agree because of our particular choice of pinning.  The inclusion of $Z$ in $\Tt$ amounts to the product $\prod_{s \in S} \la_s$, hence \eqref{tits.1} is injective and the claim is proved.
\end{proof}

%%%%%%%%%%%%%%%%%%%%%%%%%%%%%%%%%%%%%%%%%%
\section{Outer automorphisms of semisimple groups} \label{outer.sec}

We maintain the notation of the previous section, so that $G$ is semisimple over a field $k$ and $\D$ is a set of simple roots, equivalently, the Dynkin diagram of $G$.  The Galois action on $\D$ induces an action on $\Aut(\D)$, and in this way we view $\Aut(\D)$ as a finite \'etale (but not necessarily constant) group scheme.
There is a map 
\[
\alpha \!: \aut(G)(k) \ra \aut(\D)(k)
\]
described for example in \cite[\S16.3]{Sp:LAG}, and one can ask if this map is surjective.   That is, does every connected component of $\aut(G) \times \ksep$ that is defined over $k$ necessarily have a $k$-point?

One obstruction to $\alpha$ being surjective can come from the fundamental group, so we assume that $G$ is simply connected.  (One could equivalently assume that $G$ is adjoint.)
Another obstruction comes from the Tits class, as we now explain. 
There is a commutative diagram
\[
\xymatrix{
\aut(G) \ar[d] \ar[r]^\alpha & \aut(\D) \ar[dl] \\
\aut(Z)
 }
\]
where the diagonal arrow comes from the natural action of $\aut(\D)$ on the coroot lattice.  Hence $\aut(\D)(k)$ acts on $H^2(k, Z)$ and we have:

\begin{thm} \label{flayed.prop}
Recall that $G$ is assumed semisimple and simply connected.  Then there is an inclusion 
\begin{equation} \label{autset}
\im \left[ \alpha \!: \aut(G)(k) \ra \aut(\D)(k) \right] \subseteq \{ \pi \in \aut(\D)(k) \mid \pi(t_G) = t_G \}.
\end{equation}
Furthermore, the following are equivalent:
\begin{enumerate}
\renewcommand{\theenumi}{\alph{enumi}}
\item \label{c1} Equality holds in \eqref{autset}.
\item \label{c2} The sequence $H^1(k, Z) \ra H^1(k, G) \ra H^1(k, \aut(G))$ is exact.
\item \label{c3} $\ker \delta \cap \ker \left[ H^1(k, \Gb) \ra H^1(k, \aut(\Gb)) \right] = 0$.
\end{enumerate}
\end{thm}

\begin{proof}
We consider the interlocking exact sequences
\[
\begin{CD}
@. @. H^1(k, Z) \\
@. @. @VVV \\
@. @. H^1(k, G)\\
@.@.@VV{q}V \\
\aut(G)(k) @>\alpha>> \aut(\D)(k) @>{\beta}>> H^1(k, \Gb) @>\e>> H^1(k,\aut(G)) \\
@. @. @VV{\delta}V \\
@. @. H^2(k,Z)
\end{CD}
\]
The crux is to prove that
\begin{equation} \label{main.lem}
\pi(t_{G_{\beta(\pi)}}) = t_G  \quad \text{for $\pi \in \aut(\D)(k)$}.
\end{equation}
Since $\Gb$ and $\aut(G)$ are smooth, we may view
their corresponding $H^1$'s as Galois cohomology.  Put $\gamma := \beta(\pi)$, so $\gamma_\s = f^{-1} \,
^\s\!f$ for some $f \in \aut(G)(\ksep)$ and every $\s \in \Gal(k)$.  The group
$G_\gamma$ has the same $\ksep$-points as $G$, but a different
Galois action $\circ$ given by $\s \circ g = \gamma_\s \s g$ for
$g \in G(\ksep)$, $\s \in \Gal(k)$, and where juxtaposition
denotes the usual Galois action on $G$.

The map $f$ gives a
$k$-isomorphism $G_\gamma \iso G$.
Sequence \eqref{std.seq} gives a commutative diagram
  \[
  \begin{CD}
  H^1(k, \Gb_\gamma) @>{\delta_\gamma}>> H^2(k, Z) \\
  @V{f}VV @V{f}VV \\
  H^1(k, \Gb) @>{\delta}>> H^2(k, Z).
  \end{CD}
  \]
Let $\eta \in Z^1(k, \Gb_\gamma)$ be a 1-cocycle
representing $\nu_{G_\gamma}$.  Then $f(\eta)$ is a 1-cocycle in
$Z^1(F,\Gb)$ and $f$ is a $k$-isomorphism $f \!:
(G_\gamma)_\eta \iso G_{f(\eta)}$.  Since $(G_\gamma)_\eta$ is
$k$-quasi-split, we have $f(\nu_{G_\gamma}) = f(\eta) = \nu_G$.
The
commutativity of the diagram gives $f(t_{G_\gamma}) = t_G$, proving \eqref{main.lem}.

It follows that $\pi \in \aut(\D)(k)$ satisfies $\pi(t_G) = t_G$ if and only if $t_{G_{\beta(\pi)}} = t_G$, if and only if $\delta(\beta(\pi)) = 0$.  That is, in \eqref{autset}, the left side is $\ker \beta$ and the right side is $\ker \delta \beta$, which makes the inclusion in \eqref{autset} and the equivalence of \eqref{c1} and \eqref{c3} obvious.  Statement \eqref{c2} says that $\ker \e q = \ker q$, i.e., $\ker \e \cap \im q = 0$, which is \eqref{c3}.
\end{proof}

It is easy to find non-simple groups, even over $\R$, for which the inclusion \eqref{autset} is proper, because the Tits index also provides an obstruction to equality.  Here is an example to show that these are not the only obstructions, even over a number field.

\begin{eg} \label{G2G2}
Fix a prime $p$ and write $x_1, x_2$ for the two square roots of $p$ in $k := \Q(\sqrt{p})$.  For $i = 1, 2$, let $H_i$ be the group of type $G_2$ associated with the 3-Pfister quadratic form  $\phi_i := \pform{-1, -1, x_i}$.
For $G = H_1 \times H_2$, the Tits index is 
\[
\begin{picture}(3,0.55)
    % put in the horizontal lines
    \put(0.2,.25){\line(1,0){1}}
    \put(0.2,.20){\line(1,0){1}}
    \put(0.2,.30){\line(1,0){1}}
    
    % put in the horizontal circles
    \multiput(0.2,.25)(1,0){2}{\circle*{\darkrad}}

    % put in the ``greater than'' sign
    \put(0.6, .15){\makebox(0.2,0.3)[s]{$<$}}

     % put in the horizontal lines
    \put(1.7,.25){\line(1,0){1}}
    \put(1.7,.20){\line(1,0){1}}
    \put(1.7,.30){\line(1,0){1}}
    
    % put in the horizontal circles
    \multiput(1.7,.25)(1,0){2}{\circle*{\darkrad}}

    % put in the ``greater than'' sign
    \put(2.1, .10){\makebox(0.2,0.3)[s]{$<$}}
\end{picture}
\]
and $\aut(\D)(k) = \Zm2$, 
but $H_1$ is not isomorphic to $H_2$, so no $k$-automorphism of $G$ interchanges the two components.  
\end{eg}

Nonetheless, we now list ``many" cases in which equality holds in \eqref{autset}.

\begin{eg} \label{qs.eg}
If $G$ is quasi-split, then $\alpha$ maps $\aut(G)(k)$ onto $\aut(\D)(k)$ by \cite[XXIV.3.10]{SGA3.3} or \cite[31.4]{KMRT}, so equality holds in \eqref{autset}.
\end{eg}

\begin{eg} \label{cd2.eg}
If $H^1(k, G) = 0$, then trivially Th.~\ref{flayed.prop}\eqref{c2} holds.  That is, \eqref{c1}--\eqref{c3} hold for every semisimple simply connected $G$ if $k$ is local (by Kneser-Bruhat-Tits), global with no real embeddings (Kneser-Harder-Chernousov), or the function field of a complex surface (de Jong-He-Starr-Gille), and conjecturally if the cohomological dimension of $k$ is at most 2 (Serre).
\end{eg}

\begin{eg} \label{flayed.eg}
Suppose $G$ is absolutely almost simple (and simply connected).  Conditions \eqref{c1}--\eqref{c3} of the proposition hold trivially if $\aut(\D)(k) = 1$, in particular if $G$ is not of type $A$, $D$, or $E_6$ or if $G$ has type $^6\!D_4$.  Conditions \eqref{c1}--\eqref{c3} also hold:
\begin{enumerate}
\renewcommand{\theenumi}{\roman{enumi}}
\item \label{flayed.inner} \emph{if $G$ is of inner type}.  If $G$ is of inner type $A$ ($n \ge 2$), then $\aut(\D) = \Zm2$ and the nontrivial element $\pi$ acts via $z \mapsto z^{-1}$ on $Z$, hence $\pi(t_G) = -t_G$.  If $2t_G = 0$, then $G$ is $\SL_1(D)$ for $D$ a central simple algebra of degree $n+1$ such that there is an anti-automorphism $\s$ of $D$, hence $g \mapsto \s(g)^{-1}$ is a $k$-automorphism of $G$ mapping to $\pi$.  (By a theorem of Albert \cite[Th.~8.8.4]{Sch} one can even arrange for $\s$ to have order 2, hence for this automorphism of $G$ to have order 2.)

Next let $G$ be of type $^1\!D_n$ for $n \ge 5$ and suppose that the nonidentity element $\pi \in \Aut(\D)(k)$ fixes the Tits class $t_G$.  The group $G$ is isomorphic to $\Spin(A, \s, f)$ for some central simple $k$-algebra $A$ of degree $2n$ and quadratic pair $(\s, f)$ on $A$ such that the even Clifford algebra $C_0(A, \s, f)$ is isomorphic to a direct product $C_+ \times C_-$ of central simple algebras.  Since $\pi$ fixes the Tits class, the algebras $C_+$ and $C_-$ are isomorphic.  The equation
$[A] + [C_+] - [C_-] = 0$
holds in the Brauer group of $k$ by \cite[9.12]{KMRT} (alternatively, as a consequence of the fact that the cocenter is an abelian group of order 4).  Therefore, $A$ is split.
  Let $\phi \in \O(A, \s, f)(k)$ be a hyperplane reflection as in \cite[12.13]{KMRT}; it does not lie in the identity component of $\O(A, \s, f)$.  The automorphism of $\SO(A, \s, f)$ given by $g \mapsto \phi g \phi^{-1}$ lifts to an automorphism of $\Spin(A, \s, f)$ that is outer, i.e., that induces the automorphism $\pi$ on $\D$.  (To recap: given a nonzero $\pi \in \aut(\D)(k)$ that preserves the Tits class, we deduced that $A$ is split and therefore $(A,\s,f)$ has an improper isometry.  Conversely, Lemma 1b from \cite[p.~42]{Kn:GC} shows: if $\chr k \ne 2$ and $(A, \s)$ has an improper isometry, then $A$ is split and obviously such a $\pi$ exists.)

For the remaining cases, we point out merely that an outer automorphism of order 3 in the $D_4$ case exists when $t_G = 0$ by triality \cite[3.6.3, 3.6.4]{Sp:ex} and an outer automorphism of order 2 in the $E_6$ case when $t_G = 0$ is provided by the ``standard automorphism'' of a $J$-structure \cite[p.~150]{Sp:jord}.

\item \label{flayed.SU} \emph{if $G$ is the special unitary group of a hermitian form relative to a separable quadratic extension $K/k$}, i.e., $G$ is of type $^2\!A_n$ and $\res_{K/k}(t_G)$ is zero in $H^2(K, Z)$.  We leave the details in this case as an exercise.

\item \label{flayed.R} \emph{if $k$ is real closed}. By the above cases, we may assume that $G$ has type $^2\!D_n$ (for $n \ge 4$) or $\dE$.  In the first case, $\aut(\D)(k) = \Zm2$ (also for $n = 4$) and $G$ is the spin group of a quadratic form by \cite[9.14]{KMRT}, so a hyperplane reflection gives the desired $k$-automorphism.  

In case $G$ has type $\dE$, combining pages 37, 38, 119, and 120 in \cite{Jac:ex} shows that the (outer) automorphism of the Lie algebra Jacobson denotes by $t$ is defined over $k$.

\end{enumerate}
I don't know any examples of absolutely almost simple $G$ where conditions \eqref{c1}--\eqref{c3} fail.  Furthermore, in all of the examples above, every $\pi$ from the right side of \eqref{autset} is not only of the form $\alpha(f)$ for some $f \in \aut(G)(k)$, but one can even pick $f$ to have the same order as $\pi$.  
\end{eg}

%%%%%%%%%%%%%%%%%%%%%%%%%%%%%%%%%%%%%%%%%%%%%
\section{Groups of type $D_\even$ over local fields} \label{local.sec}

The main point of this section is to prove the following lemma.

\begin{lem} \label{local}
Let $G$ be an adjoint semisimple group over a field $k$, and fix a maximal $k$-torus $T$ in $G$.  If $z_1, z_2$ are in the image of the map $H^1(k, T) \ra H^1(k, G)$ such that 
\begin{enumerate}
\item $G_{z_1}$ and $G_{z_2}$ are both quasi-split; or 
\item $T$ contains a maximal $k$-split torus in both $G_{z_1}$ and $G_{z_2}$ and 
  \begin{enumerate}
  \item $k$ is real closed, or
  \item $k$ is a (non-archimedean) local field and $G$ has type $D_{2n}$ for some $n \ge 2$,
  \end{enumerate}
\end{enumerate}
then $z_1 = z_2$.
\end{lem}

\begin{proof}
For short, we write $G_i$ for $G_{z_i}$.  In case (1), the uniqueness of the class $\nu_G \in H^1(k, G)$ such that $G_{\nu_G}$ is quasi-split (already used in \S\ref{folk.sec}) gives that $z_1 = \nu_G = z_2$.  So suppose (2) holds.
As $T$ is contained in both these groups, their Tits indexes are naturally identified over $k$.  In particular, if one is quasi-split then so is the other, and we are done as in (1).  So we assume that neither group is quasi-split.

In case (2a), where $k$ is real closed, one immediately reduces to the case where $G$ is absolutely simple.  That case is trivial because the isomorphism class of an adjoint simple group is determined by its Tits index, so $G_1$ is isomorphic to $G_2$.  The Tits index also determines the Tits algebras---see pages 211 and 212 of \cite{Ti:R} for a recipe---so by Prop.~\ref{tits.prop}, $\delta(z_1) = \delta(z_2)$.  The claim now follows from Example \ref{flayed.eg}\eqref{flayed.R} and Theorem \ref{flayed.prop}\eqref{c3}.

So assume for the remainder of the proof that (2b) holds.  In particular, $\delta$ is injective.  
Number the simple roots of $G_1$ with respect to $T$ as in \cite{Bou:g4}.  If $G_1$ has type $\iiD$, we take $\alpha_1$ to be the root at the end of the Galois-fixed arm of the Tits index.  Otherwise, we assign the numbering arbitrarily in case there is ambiguity (e.g., $\alpha_{2n-1}$ and $\alpha_{2n}$).
Note that $G_1$ cannot have type $\iiiD$ or $\viD$, because it is not quasi-split.

As $2\omega_i$ is in the root lattice for every $i$, the Tits algebras $\omega_i(t_{G_1})$ for $i = 2n-1, 2n$ define up to $k$-isomorphism a quaternion (Azumaya) algebra $D$ over a quadratic \'etale $k$-algebra $\ell$.  By the exceptional isomorphism $D_2 = A_1 \times A_1$ and a Tits algebra computation, $\PGL_1(D)$ is isomorphic to $\PSO(M_2(H), \s, f)$ for $H$ the quaternion algebra underlying $\omega_1(t_{G_1})$ and some quadratic pair $(\s, f)$ such that the even Clifford algebra $C_0(\s, f)$ is isomorphic to $D$, cf.~\cite[15.9]{KMRT}.  Appending $2n-2$ hyperbolic planes to $(\s, f)$, we obtain a quadratic pair $(\s_0, f_0)$ such that $C_0(\s_0, f_0)$ is Brauer-equivalent to $D$.  As $\PSO(M_{2n}(H), \s_0, f_0)$ has the same Tits algebras as $G_1$ (up to renumbering the simple roots of $G_1$), Prop.~\ref{tits.prop} and injectivity of $\delta$ implies that the two groups are isomorphic.  (We have just given a characteristic-free proof of Tsukamoto's theorem \cite[10.3.6]{Sch}, relying on the Bruhat-Tits result that $\delta$ is injective.)

Now both $G_1$ and $G_2$ have the same Tits index and semisimple anisotropic kernels of Killing-Cartan type a product of $A_1$'s.  As there is a unique quaternion division algebra over each finite extension of $k$, it follows that $G_1$ and $G_2$ have the same Tits class, i.e., $\delta(z_1) = \delta(z_2)$.
\end{proof}

In the statement of (2b), we cannot replace ``$D_{2n}$ for some $n \ge 2$" with ``$D_\ell$ for some $\ell$" because the claim fails for groups of type $D_{\text{odd}}$.  
This can been seen already for type $D_3 = A_3$: one can find $z_1, z_2 \in H^1(k, \PGL_4)$ so that $G_1$ and $G_2$ are both isomorphic to  $\aut(B)^\circ$ for a division algebra $B$ of degree 4, but $\delta(z_1) = -\delta(z_2)$ in $H^2(k, \mu_4) = \Zm4$.  Adding hyperbolic planes as in the proof of the lemma gives a counterexample for all odd $\ell$.   This counterexample is visible in the proof: for groups $G_1, G_2$ of type $D_\ell$ with $\ell$ odd and $\ge 3$, the semisimple anisotropic kernels have Killing-Cartan type a product of $A_1$'s and an $A_3$ and the very last sentence of the proof fails.

%%%%%%%%%%%%%%%%%%%%%%%%%%%%%%%%%%%%%%%%%%%%%
\section{Groups of type $D_\even$ over global fields} \label{global.sec}

The following technical theorem concerning groups over a global field connects our Theorem \ref{flayed.prop} (about groups over an arbitrary field) with the results in \cite{PrRap:weakly}.  It  implies Theorem 9.1 of \cite{PrRap:Deven}.

\begin{thm} \label{MT}
Let $G_1$ and $G_2$ be adjoint groups of type $D_{2n}$ for some $n \ge 2$ over a global field $K$, such that $G_1$ and $G_2$ have the same quasi-split inner form---i.e., the smallest Galois extension of $K$ over which $G_1$ is of inner type is the same as for $G_2$.  If there exists a maximal torus $T_i$ in $G_i$ for $i = 1$ and $2$ such that
\begin{enumerate}
\item there is a $K_\sep$-isomorphism $\phi \!: G_1 \ra G_2$ whose restriction to $T_1$ is a $K$-isomorphism $T_1 \ra T_2$; and 
\item there is a finite set $V$ of places of $K$ such that:
\begin{enumerate}
\item For all $v \not\in V$, $G_1$ and $G_2$ are quasi-split over $K_v$.
\item For all $v \in V$, $(T_i)_{K_v}$ contains a maximal $K_v$-split torus of $(G_i)_{K_v}$;
\end{enumerate}
\end{enumerate}
then $G_1$ and $G_2$ are isomorphic over $K$.
\end{thm}

The hypotheses are what one obtains by assuming the existence of weakly commensurable arithmetic subgroups, see for example Theorems 1 and 6 and Remark 4.4---and especially p.156---in \cite{PrRap:weakly}.  Note that the groups appearing in the theorem can be trialitarian, i.e., of type $\iiiD$ or $\viD$.  We remark that Bruce Allison gave an isomorphism criterion with very different hypotheses in \cite[Th.~7.7]{A:num}.

\begin{proof}
Write $G$ for the unique adjoint quasi-split group that is an inner form of $G_1$ and $G_2$.  By Steinberg \cite[pp.~338, 339]{PlatRap}, there is a $K_\sep$-isomorphism $\psi_2 \!: G_2 \ra G$ whose restriction to $T_2$ is defined over $K$.  We put 
$\psi_1 := \psi_2 \phi$ and
$T := \psi_2(T_2) = \psi_1(T_1)$.
Then $G_i$ is isomorphic to $G$ twisted by the 1-cocycle $\s \mapsto \psi_i (^\s \psi_i)^{-1}$.  But this 1-cocycle consists of elements of $\aut(G)$ that fix $T$ elementwise, hence belong to $T$ itself.  That is, for $i = 1, 2$, there is a cocycle $z_i$ in the image of $H^1(K, T) \ra H^1(K, G)$ such that $G_i$ is isomorphic to $G$ twisted by $z_i$.\footnote{This argument does not use the fact that $K$ is a number field nor that $G_1$ and $G_2$ have type $D_{2n}$, so roughly speaking it applies generally to the situation where $G_1$ and $G_2$ share a maximal torus over the base field---more precisely, to the situation arising in Remark 4.4 of \cite{PrRap:weakly}.}

Now Lemma \ref{local} gives that $\res_{K_v/K}(z_1) = \res_{K_v/K}(z_2)$ for every $v$, hence $z_1 = z_2$ by the Kneser-Harder Hasse Principle \cite[p.~336, Th.~6.22]{PlatRap} and $G_1$ is isomorphic to $G_2$ over $K$.
\end{proof}

\medskip

\noindent{\small{\textbf{Acknowledgments}  My research was partially supported by NSF grant DMS-0653502.  I thank the University of Virginia for its hospitality during a visit there where this work began, and Andrei Rapinchuk, Gopal Prasad, and Emily Hamilton for their comments and suggestions.}}

%\bibliographystyle{amsalpha}
%\bibliography{skip_master}

\providecommand{\bysame}{\leavevmode\hbox to3em{\hrulefill}\thinspace}
\providecommand{\MR}{\relax\ifhmode\unskip\space\fi MR }
\providecommand{\MRhref}[2]{%
  \href{http://www.ams.org/mathscinet-getitem?mr=#1}{#2}
}
\providecommand{\href}[2]{#2}

\end{document}